\documentclass[11pt]{amsart}
\usepackage[T1]{fontenc}
\usepackage{lmodern}
\usepackage{amssymb}
\usepackage{microtype}
\usepackage[hidelinks]{hyperref}
\numberwithin{equation}{section}
\hypersetup{
  pdftitle={d = omega\_1 implies a = omega\_1},
  pdfauthor={Jos\'e de Jes\'us Pelayo-G\'omez},
  pdfsubject={Domination, almost disjointness, and preservation at omega\_1},
  pdfkeywords={Cardinal invariants, dominating number, maximal almost disjoint family, Roitman's problem, bounding forcing, indestructibility}
}

\newtheorem{theorem}{Theorem}[section]
\newtheorem{lemma}[theorem]{Lemma}
\newtheorem{corollary}[theorem]{Corollary}

\newcommand{\om}{\omega}
\newcommand{\omo}{\omega_1}
\newcommand{\aaa}{\mathfrak a}
\newcommand{\bb}{\mathfrak b}
\newcommand{\dd}{\mathfrak d}
\newcommand{\A}{\mathcal A}
\newcommand{\leas}{\leq^*}

\title{$\dd=\omo\text{ implies }\aaa=\omo$}
\author{Jos\'e de Jes\'us Pelayo-G\'omez}
\address{Independent researcher}
\email{pelayuss@gmail.com}
\date{}

\begin{document}

\begin{abstract}
We prove in ZFC that $\mathfrak d=\omega_1$ implies $\mathfrak a=\omega_1$,
settling a question of Roitman from the 1970s. Here $\mathfrak d$ is the
least size of a family that eventually dominates every function in
$\omega^\omega$, and $\mathfrak a$ is the least size of an infinite
maximal almost disjoint (MAD) family of infinite subsets of $\omega$.
Consequently, $\mathfrak a\leq\mathfrak d$ whenever the continuum
$\mathfrak c=2^{\aleph_0}$ is at most $\omega_2$. This answers
negatively Shelah's question whether $\mathfrak d<\mathfrak a$ is
consistent with $\mathfrak c=\omega_2$, and shows that his model of
$\mathfrak d=\omega_2<\mathfrak a=\mathfrak c=\omega_3$ has the least
possible values of $\mathfrak d$, $\mathfrak a$ and $\mathfrak c$ among
models of $\mathfrak d<\mathfrak a$.
The proof also gives a preservation theorem: from any dominating
family of size $\omega_1$, we construct a MAD family of size
$\omega_1$ that remains maximal in every outer model in which the
given family is still dominating. Thus, under $\mathfrak d=\omega_1$,
some MAD family is indestructible by every $\omega^\omega$-bounding
forcing; it can moreover be chosen almost strongly separable, hence
Cohen-indestructible. These results give negative answers, under
$\mathfrak d=\omega_1$, to two questions of Hru\v{s}\'ak and to one
of Brendle, Guzm\'an, Hru\v{s}\'ak and Raghavan.
\end{abstract}

\maketitle

\section{Introduction}

An infinite family $\A\subseteq[\om]^\om$ is \emph{maximal almost
disjoint}, or \emph{MAD}, if its members have pairwise finite
intersections and every infinite subset of $\om$ has infinite
intersection with some member of $\A$. The almost disjointness
number $\aaa$ is the least size of a MAD family. The dominating
number $\dd$ is the least size of a subset of $\om^\om$ that is
cofinal in the eventual domination order. We write
$\mathfrak c=2^{\aleph_0}$.

The bounds $\omo\leq\aaa\leq\mathfrak c$ and
$\omo\leq\dd\leq\mathfrak c$ show that the
continuum hypothesis (CH) implies $\aaa=\dd=\omo$.
The hypothesis $\dd=\omo$ provides a dominating family of size
$\omo$ without requiring $[\om]^\om$ to have that size.
Roitman asked whether $\dd=\omo$ nevertheless implies $\aaa=\omo$.
Brendle dates the question to the 1970s \cite[Problem~5]{B13}.
It is Problem~12 in \emph{Open Problems in Topology II}
\cite[p.~183]{HS07} and Problem~4.1 in Miller's problem list
\cite{M15}, and it is stated in \cite[p.~351]{H01c},
\cite[Question~9]{RS12}, \cite[Problem~1.1]{GH22} and
\cite[Problem~1.1]{BGHR26}.
We give a positive answer.

\begin{theorem}\label{thm:main}
In ZFC, $\dd=\omo$ implies $\aaa=\omo$.
\end{theorem}

Shelah proved the consistency of
$\dd=\om_2<\aaa=\mathfrak c=\om_3$ without large cardinals
\cite[Lemma~3.1]{S04}, and asked whether $\dd<\aaa$ is consistent
with $\mathfrak c=\om_2$ \cite[p.~69, Question~10.1(1)]{S00},
\cite[Problem~0.1]{S04}, \cite[Problem~0.1]{S21}. The following
consequence answers this question negatively and shows that his model
has the least possible values of $\dd$, $\aaa$ and $\mathfrak c$ among
models of $\dd<\aaa$.

\begin{corollary}\label{cor:shelah}
If $\dd\in\{\omo,\mathfrak c\}$, then $\aaa\leq\dd$.
In particular, $\mathfrak c\leq\om_2$ implies $\aaa\leq\dd$.
Every model of $\dd<\aaa$ therefore satisfies
$\dd\geq\om_2$ and $\mathfrak c\geq\aaa\geq\om_3$, and
$\dd=\om_2<\aaa=\mathfrak c=\om_3$ is consistent relative to ZFC.
\end{corollary}

\begin{proof}
The case $\dd=\omo$ is Theorem~\ref{thm:main}, and the case
$\dd=\mathfrak c$ follows from $\aaa\leq\mathfrak c$.
If $\mathfrak c\leq\om_2$, these are the only possibilities for
$\dd$. In any model of $\dd<\aaa$, the first assertion gives
$\omo<\dd<\mathfrak c$, so $\dd\geq\om_2$ and
$\mathfrak c\geq\aaa\geq\dd^+\geq\om_3$.
Applying \cite[Lemma~3.1]{S04} over a
model of GCH, with Shelah's parameters $\mu=\om_2$ and $\lambda=\om_3$, gives
$\dd=\om_2<\aaa=\mathfrak c=\om_3$.
\end{proof}

Vaughan asked whether $\mathfrak i<\aaa$ is consistent, where
$\mathfrak i$ is the least size of a maximal independent family
\cite[Problem~1.1(c)]{V90}; see also \cite[Question~2]{FS22}.
Shelah's inequality $\dd\leq\mathfrak i$
\cite[Appendix, Theorem~A.1]{V90} gives the following restriction.

\begin{corollary}\label{cor:vaughan}
If $\mathfrak i=\omo$, then $\aaa=\omo$. If $\mathfrak c\leq\om_2$,
then $\aaa\leq\mathfrak i$. Every model of $\mathfrak i<\aaa$ satisfies
$\om_2\leq\dd\leq\mathfrak i<\aaa\leq\mathfrak c$.
\end{corollary}

\begin{proof}
If $\mathfrak i=\omo$, then $\dd=\omo$, so Theorem~\ref{thm:main}
applies. The remaining assertions follow from
$\dd\leq\mathfrak i$ and Corollary~\ref{cor:shelah}.
\end{proof}

The first assertion supplies the case $\kappa=\om$ of the implication
$\mathfrak i(\kappa)=\kappa^+\Rightarrow\aaa(\kappa)=\kappa^+$,
proved for regular uncountable $\kappa$ by Fischer and Montoya
\cite[Corollary~2.5]{FM22}.

Earlier results connected the existence of small MAD families
with guessing principles. Hru\v{s}\'ak proved that
$\diamondsuit_{\dd}$ implies $\aaa=\omo$
\cite[Proposition~I.2]{H01}. This guessing principle is strictly
stronger than $\dd=\omo$: already $\mathrm{CH}$ is consistent
with its failure \cite[Theorem~II.6]{H01}. Theorem~\ref{thm:main} thus
removes the guessing hypothesis from this result.
Moore, Hru\v{s}\'ak and D\v{z}amonja later proved that the parametrized
diamond $\diamondsuit(\bb)$ implies $\aaa=\omo$
\cite[Theorem~7.2]{MHD04}, where $\bb$ is the least size of an
unbounded family in $\om^\om$. Their hypothesis holds in the
Miller model, where $\dd>\omo$, but fails in some model of CH
\cite[Section~8]{MHD04}. Thus $\diamondsuit(\bb)$ and $\dd=\omo$
are incomparable hypotheses.

Under the hypothesis $\dd=\omo$ itself, Raghavan and Shelah
constructed a van Douwen MAD family of functions that is the
union of $\omo$ compact sets \cite[Theorem~1]{RS12}.
That result bounds the number of compact pieces rather than the
cardinality of their union. Here we construct $\omo$ individual
infinite sets; compact families serve as classes of candidates
that a single stage handles simultaneously.

We build $\A=\{a_\xi:\xi<\omo\}$ by recursion, adding one set
at each stage. We normalize a dominating family to obtain a pointwise
dominating sequence $\langle d_\alpha:\alpha<\omo\rangle$ with
$d_\alpha(n)\geq n$ for all $\alpha$ and $n$.
Each $d_\alpha$ determines a compact family $K_\alpha$ of infinite
sets whose increasing enumerations are bounded pointwise by $d_\alpha$.
Given an infinite set $y$, a first application of domination gives
$\alpha$ with $y\in K_\alpha$. If $y$ meets some $a_\eta$ with
$\eta<\alpha$ in an infinite set, it is already captured. Otherwise,
a second application gives $\beta$ such that $d_\beta$ bounds its
finite intersections with these earlier members. Removing from $\om$
the tails of those members beyond the prescribed bounds gives a set
$G_{\alpha,\beta}$ containing $y$. We assign the pair $(\alpha,\beta)$
in advance to a stage $\Gamma(\alpha,\beta)\geq\alpha$, when all the
members used to define $G_{\alpha,\beta}$ are available.

Each $d_\alpha$ also determines a partition of $\om$ into finite
intervals $I^\alpha_j$ met by every member of $K_\alpha$.
At stage $\Gamma(\alpha,\beta)$, we arrange that the new member
contains $G_{\alpha,\beta}\cap I^\alpha_j$ for infinitely many $j$
whenever $G_{\alpha,\beta}$ meets every interval of this partition.
It then has infinite intersection with every member of $K_\alpha$
contained in $G_{\alpha,\beta}$, including $y$.
Thus $\omo$ stages suffice without an enumeration of $[\om]^\om$.

The new member must be almost disjoint from the members already
constructed. The intersection bound handles the members with
$\eta<\alpha$. To handle the remaining members, from stage $\alpha$
onward we maintain an infinite set of reserved interval indices.
These sets decrease modulo finite sets as the construction proceeds.
At any stage $\xi\geq\alpha$, each $a_\eta$ with
$\alpha\leq\eta<\xi$ meets only finitely many of the intervals
still reserved for $\alpha$. Taking the portions of
$G_{\alpha,\beta}$ in infinitely many such intervals therefore
preserves almost disjointness. Between successive chosen intervals
we leave further intervals unused, maintaining the reservations
needed at later stages.

The construction uses the countability of every proper initial
segment of $\omo$ in two places. First, only countably many
reservations are active at stage $\xi<\omo$: a single sequence of
choices serves all of them, and at limit stages each has a
pseudo-intersection. Second, countability makes each intersection
trace a function in $\om^\om$, to which domination applies.
These steps do not extend directly to uncountable stages;
the general inequality $\aaa\leq\dd$ fails in Shelah's model.

The construction works for an arbitrary sequence of $\omo$ functions;
domination is used only to prove maximality. Since maximality is
witnessed by intervals fixed during the recursion, the proof of
Theorem~\ref{thm:main} also gives a preservation theorem.
For transitive models $V$ and $W$ of ZFC, we call $W$ an \emph{outer
model} of $V$ if $V\subseteq W$ and they have the same ordinals.
In particular, every forcing extension of $V$ is an outer model of $V$.

\begin{theorem}[Preservation]\label{thm:preservation}
Let $D_0\subseteq\om^\om$ be a dominating family of size $\omo$.
There is a MAD family $\A$ of size $\omo$ which remains MAD in
every outer model in which $D_0$ remains dominating.
In particular, $\A$ is indestructible by every
$\om^\om$-bounding forcing.
\end{theorem}

Since $D_0$ dominates in the ground model, it remains dominating
in an outer model if and only if the ground-model functions dominate
there. Equivalently, the outer model adds no function unbounded
over the ground model. This hypothesis cannot be weakened to adding
no dominating reals: by a theorem of Shelah (see
\cite[Theorem~4.6(5)]{BGHR26}), under CH every MAD family is
destroyed by some proper forcing that adds no dominating reals.

Theorem~\ref{thm:preservation} bears on the route by which a negative
answer to Roitman's question was expected. Shelah suggested obtaining
$\dd<\aaa$ with $\mathfrak c=\om_2$ by a countable support iteration
of proper $\om^\om$-bounding forcings over a model of CH, each
destroying one MAD family. By his iteration theorems, it suffices to
prove assertion~$(*)$ of \cite[pp.~21--22]{S00}: every MAD family is
destroyed by such a forcing of size $\omo$. In proving~$(*)$ one may
assume CH or even $\diamondsuit_{\omo}$, and replace the size bound
by the $\aleph_2$-properness isomorphism condition. Shelah later noted
that such an iteration keeps $\dd=\omo$, but that the forcing needed
for a single step was not known to exist \cite[p.~188]{S04}.
Hru\v{s}\'ak, observing that Shelah's preservation theorems set up
the framework for the anticipated negative solution, asked for the
same step without the size bound: assuming CH, can every MAD family
be destroyed by a proper $\om^\om$-bounding forcing
\cite[p.~351]{H01c}? Guzm\'an and Hru\v{s}\'ak restated this as
\cite[Problem~1.2]{GH22}. Brendle, Guzm\'an, Hru\v{s}\'ak and
Raghavan asked it without assuming CH \cite[Problem~1.3]{BGHR26},
as did Chodounsk\'y and Guzm\'an in terms of forcings preserving
dominating families \cite[p.~14]{CG21}.

Two earlier preservation results come close to a negative answer to
this question. Under CH, Hru\v{s}\'ak proved that for each proper
$\om^\om$-bounding forcing $\mathbb P$ of size $\omo$ there is a
$\mathbb P$-indestructible MAD family \cite[Lemma~III.1]{H01S}.
Guzm\'an and Hru\v{s}\'ak proved that Shelah--Stepr\=ans MAD families
survive every strategically bounding forcing
\cite[Corollary~4.7(1)]{GH22}; such families exist under CH
\cite[Proposition~4.4]{GH22}.

Theorem~\ref{thm:preservation} gives, under $\dd=\omo$, a single MAD
family that survives every $\om^\om$-bounding forcing, whether or not
it is proper and whatever its size. It therefore answers
Hru\v{s}\'ak's question and its restatement by Guzm\'an and
Hru\v{s}\'ak negatively, and its versions without CH have a
negative answer in every model of $\dd=\omo$. Since Shelah's
iteration keeps $\dd=\omo$, Theorem~\ref{thm:main} already shows that it cannot produce a model
of $\dd<\aaa$. Theorem~\ref{thm:preservation} locates the obstruction
at a single step: assertion~$(*)$ fails in every model of CH, and so
in each of the forms described above.

Hru\v{s}\'ak also asked whether every Cohen-indestructible MAD family
can be destroyed by some $\om^\om$-bounding forcing, not necessarily
proper, noting that for Roitman's question it would suffice to answer
this under CH \cite[Question~9]{H01c}. The family of
Theorem~\ref{thm:preservation} can be chosen Cohen-indestructible;
in fact almost strongly separable, hence Shelah--Stepr\=ans and tight
(Theorem~\ref{thm:cohen}). This answers his question negatively
in every model of $\dd=\omo$.
The family is thus a Shelah--Stepr\=ans MAD family that survives
every $\om^\om$-bounding forcing, not only the strategically bounding
ones considered in \cite{GH22}.

Section~\ref{sec:prelim} contains an elementary lemma and the
notation for traces, Section~\ref{sec:construction} the recursion,
Section~\ref{sec:maximality} the proofs of Theorems~\ref{thm:main}
and~\ref{thm:preservation}, and Section~\ref{sec:ch} the almost
strongly separable refinement.

\section{Intervals and traces}\label{sec:prelim}

We write $\om=\{0,1,\ldots\}$. Intervals of integers have their usual
meaning; in particular $[0,m)$ has $m$ elements. For an infinite set
$x\subseteq\om$, let $e_x\colon\om\to\om$ be its increasing enumeration.
For $f,g\in\om^\om$, the notation $f\leq g$ means pointwise domination,
and $f\leas g$ means eventual domination. We write $f=^*g$ when $f$
and $g$ differ at only finitely many coordinates. For subsets of
$\om$, $x\subseteq^*z$ means that $x\setminus z$ is finite.

Put $\Delta=\{d\in\om^\om:(\forall n)\ d(n)\geq n\}$.
For $d\in\Delta$, define
\[
 K_d=\{x\subseteq\om:(\forall n)\
                 |x\cap[0,d(n)]|\geq n+1\}.
\]
Thus $K_d=\{x\in[\om]^\om:e_x\leq d\}$. The defining conditions are
clopen in $2^\om$, so $K_d$ is compact; it is also upward closed
under inclusion. Define
\[
 h_d(0)=0,\qquad h_d(j+1)=d(h_d(j))+1,\qquad
 I^d_j=[h_d(j),h_d(j+1)).
\]
Since $d(n)\geq n$, the intervals $I^d_j$ form a partition of $\om$
into nonempty finite intervals.

\begin{lemma}\label{lem:intervals}
Every member of $K_d$ meets every interval $I^d_j$.
\end{lemma}

\begin{proof}
Let $h=h_d(j)$ and $x\in K_d$. There are at least $h+1$ points of $x$
in $[0,d(h)]$ and only $h$ integers below $h$. Hence
$x\cap[h,d(h)]\ne\varnothing$.
\end{proof}

For a sequence $\mathbf c=\langle c_k:k\in\om\rangle$ of subsets of
$\om$ and $g\in\om^\om$, put
\[
 G(\mathbf c,g)=
 \om\setminus\bigcup_{k\in\om}\bigl(c_k\setminus[0,g(k))\bigr).
\]
Then
\begin{equation}\label{eq:traces}
 x\subseteq G(\mathbf c,g)
 \quad\Longleftrightarrow\quad
 (\forall k)\ x\cap c_k\subseteq[0,g(k)).
\end{equation}
In particular, $G(\mathbf c,g)$ has finite intersection with each
$c_k$. If $x\cap c_k$ is finite for every $k$, define its trace
by $t_x(k)=1+\max(x\cap c_k)$ when $x\cap c_k\ne\varnothing$,
and $t_x(k)=0$ otherwise.
Whenever $t_x\leq g$, equation~\eqref{eq:traces} gives
$x\subseteq G(\mathbf c,g)$.

\section{The construction}\label{sec:construction}

Fix an arbitrary sequence $\langle d_\alpha:\alpha<\omo\rangle$
in $\Delta$, and abbreviate $K_\alpha=K_{d_\alpha}$ and
$I^\alpha_j=I^{d_\alpha}_j$. No domination assumption is needed in
this section.

Order $\omo\times\omo$ first by the maximum coordinate and then
lexicographically, and let $\Gamma(\alpha,\beta)$ be the order type
of the set of predecessors of $(\alpha,\beta)$. This well-order has
size $\aleph_1$ and every proper initial segment of it is countable,
so $\Gamma$ is a bijection from $\omo\times\omo$ onto $\omo$. Moreover,
$\Gamma(\alpha,\beta)\geq\alpha$, because the pairs $(\gamma,0)$ with
$\gamma<\alpha$ precede $(\alpha,\beta)$.

The following lemma is the combinatorial core of the argument.
It states the simultaneous capture property that we shall obtain
by recursion.

\begin{lemma}[Construction lemma]\label{lem:construction}
There is an almost disjoint family
$\A=\{a_\xi:\xi<\omo\}$ of $\omo$ distinct infinite sets,
together with, for each $\alpha<\omo$, an enumeration
$\mathbf c^\alpha=\langle c^\alpha_k:k\in\om\rangle$ of
$\{a_\eta:\eta<\alpha\}$, possibly with repetitions
(with the convention $c^0_k=\varnothing$ for every $k$), such that the
following holds for all $\alpha,\beta<\omo$. If
$G(\mathbf c^\alpha,d_\beta)$ meets all but finitely many intervals
$I^\alpha_j$, then $a_{\Gamma(\alpha,\beta)}$ contains
$G(\mathbf c^\alpha,d_\beta)\cap I^\alpha_j$ for infinitely many $j$.
In particular, for every $x\subseteq\om$,
\[
 x\in K_\alpha,\quad x\subseteq G(\mathbf c^\alpha,d_\beta)
 \quad\Longrightarrow\quad
 |x\cap a_{\Gamma(\alpha,\beta)}|=\om.
\]
\end{lemma}

The last assertion follows from the first. By
Lemma~\ref{lem:intervals}, such an $x$ meets every interval
$I^\alpha_j$, and it does so inside $G(\mathbf c^\alpha,d_\beta)$.

In terms of \eqref{eq:traces}, stage $\Gamma(\alpha,\beta)$ captures
every member of $K_\alpha$ whose intersections with the sets
$c^\alpha_k$ are bounded by $d_\beta$.
The rest of this section proves Lemma~\ref{lem:construction} by a
recursion defining $a_\xi$ at stage $\xi$. Write
$\A_\xi=\{a_\eta:\eta<\xi\}$. At the beginning of stage $\alpha$,
fix an enumeration $\mathbf c^\alpha$ of $\A_\alpha$ as in the lemma.
For $\alpha,\beta<\omo$, put
$G_{\alpha,\beta}=G(\mathbf c^\alpha,d_\beta)$;
in particular $G_{0,\beta}=\om$.

For $\alpha\leq\xi$ we also maintain sets $R^\alpha_\xi\subseteq\om$
of \emph{reserved} indices. Besides almost disjointness, at stage
$\xi$ we maintain the reservation invariant $(\mathrm{R}_\xi)$: the
sets $R^\alpha_\eta$ with $\alpha\leq\eta\leq\xi$ are infinite,
\begin{equation}\label{eq:tower}
 R^\alpha_\alpha=\om,\qquad
 R^\alpha_\zeta\subseteq^*R^\alpha_\eta
 \qquad(\alpha\leq\eta\leq\zeta\leq\xi),
\end{equation}
and
\begin{equation}\label{eq:reserved}
 a_\eta\cap I^\alpha_j=\varnothing
 \qquad(\alpha\leq\eta<\xi,\ j\in R^\alpha_{\eta+1}).
\end{equation}
Thus, for each $\alpha$, the sets $R^\alpha_\eta$ form a
$\subseteq^*$-decreasing family, and every member added from stage
$\alpha$ onward misses the intervals of $d_\alpha$ reserved after it.

\begin{lemma}\label{lem:reserved}
Assume $(\mathrm{R}_\xi)$. If $\alpha\leq\eta<\xi$, then $a_\eta$
meets only finitely many intervals $I^\alpha_j$ with
$j\in R^\alpha_\xi$.
\end{lemma}

\begin{proof}
By \eqref{eq:tower}, $R^\alpha_\xi\setminus R^\alpha_{\eta+1}$ is
finite, and by \eqref{eq:reserved}, $a_\eta$ misses $I^\alpha_j$ for
every $j\in R^\alpha_{\eta+1}$.
\end{proof}

We now describe stage $\xi$. Suppose that $\A_\xi$ is almost
disjoint and $(\mathrm{R}_\xi)$ holds. Let
$(\alpha,\beta)=\Gamma^{-1}(\xi)$, so $\alpha\leq\xi$.
If $G_{\alpha,\beta}$ meets $I^\alpha_j$ for infinitely many
$j\in R^\alpha_\xi$, set $(\rho,H)=(\alpha,G_{\alpha,\beta})$;
otherwise set $(\rho,H)=(0,\om)$. In either case the set
$J=\{j\in R^\rho_\xi:H\cap I^\rho_j\ne\varnothing\}$ is infinite.
Enumerate the countable set $\xi+1$ as
$\langle\gamma_n:n\in\om\rangle$ so that each ordinal occurs
infinitely often.

At each finite step we choose an interval contributing to $a_\xi$
and a later interval to reserve for $\gamma_n$.
Starting with $b_0=0$, choose $j_n\in J$ and
$i_n\in R^{\gamma_n}_\xi$ such that
\[
 b_n\leq\min I^\rho_{j_n},\qquad
 \max I^\rho_{j_n}<\min I^{\gamma_n}_{i_n},\qquad
 b_{n+1}=\max I^{\gamma_n}_{i_n}+1.
\]
These choices are possible because both sets of available indices
are infinite. Define $a_\xi=\bigcup_{n\in\om}(H\cap I^\rho_{j_n})$ and set
\[
 R^\gamma_{\xi+1}=\{i_n:\gamma_n=\gamma\}\quad(\gamma\leq\xi),\qquad
 R^{\xi+1}_{\xi+1}=\om.
\]

\begin{lemma}\label{lem:stage}
The set $a_\xi$ is infinite and almost disjoint from every member
of $\A_\xi$. The invariant $(\mathrm{R}_{\xi+1})$ holds.
Let $(\alpha,\beta)=\Gamma^{-1}(\xi)$. If $G_{\alpha,\beta}$ meets
all but finitely many intervals $I^\alpha_j$, then $a_\xi$ contains
$G_{\alpha,\beta}\cap I^\alpha_j$ for infinitely many $j$.
\end{lemma}

\begin{proof}
By construction,
\[
 b_n\leq\min I^\rho_{j_n}\leq\max I^\rho_{j_n}
       <\min I^{\gamma_n}_{i_n}\leq\max I^{\gamma_n}_{i_n}<b_{n+1}.
\]
Thus the intervals $[b_n,b_{n+1})$ are pairwise disjoint, the $j_n$
are distinct, and
\[
 a_\xi\cap[b_n,b_{n+1})=H\cap I^\rho_{j_n}\ne\varnothing,
 \qquad a_\xi\cap I^{\gamma_n}_{i_n}=\varnothing.
\]
In particular, $a_\xi$ is infinite.

Let $\eta<\xi$. If $\eta<\rho$, then $\rho=\alpha$ and
$H=G_{\alpha,\beta}$. Since $a_\eta=c^\alpha_k$ for some $k$,
\eqref{eq:traces} shows that $a_\xi\subseteq H$ has finite
intersection with $a_\eta$. If $\rho\leq\eta<\xi$,
Lemma~\ref{lem:reserved} shows that $a_\eta$ meets only finitely
many intervals $I^\rho_j$ with $j\in R^\rho_\xi$. The set $a_\xi$
is contained in the union of the distinct finite intervals
$I^\rho_{j_n}$, where $j_n\in R^\rho_\xi$, so again
$a_\eta\cap a_\xi$ is finite.

To check $(\mathrm{R}_{\xi+1})$, note that each $\gamma\leq\xi$
occurs as $\gamma_n$ infinitely often and that the corresponding
$i_n$ are distinct. Hence $R^\gamma_{\xi+1}\subseteq R^\gamma_\xi$
is infinite, which gives \eqref{eq:tower}. The display above gives
\eqref{eq:reserved} for $\eta=\xi$.

Finally, suppose that $G_{\alpha,\beta}$ meets all but finitely many
intervals $I^\alpha_j$. Then it meets $I^\alpha_j$ for infinitely many
$j\in R^\alpha_\xi$. Thus $(\rho,H)=(\alpha,G_{\alpha,\beta})$, and
$a_\xi\supseteq G_{\alpha,\beta}\cap I^\alpha_{j_n}$
for every $n$.
\end{proof}

\begin{proof}[Proof of Lemma~\ref{lem:construction}]
Almost disjointness and $(\mathrm{R}_0)$ hold trivially at stage
$0$, with $R^0_0=\om$. Lemma~\ref{lem:stage} supplies the successor
step. At a limit stage $\xi$, almost disjointness and
\eqref{eq:reserved} are inherited. For $\alpha<\xi$, choose an
increasing sequence $\langle\eta_k:k\in\om\rangle$ cofinal in
$[\alpha,\xi)$ and let $R^\alpha_\xi$ consist of the $k$-th element
of $R^\alpha_{\eta_0}\cap\cdots\cap R^\alpha_{\eta_k}$, for
$k\in\om$. These finite intersections are infinite by
\eqref{eq:tower}, so $R^\alpha_\xi$ is an infinite
pseudo-intersection of the sets $R^\alpha_\eta$ with
$\alpha\leq\eta<\xi$. Put $R^\xi_\xi=\om$. This gives
$(\mathrm{R}_\xi)$.
Thus the recursion produces $\omo$ pairwise almost disjoint
infinite sets, which are necessarily distinct. The property of
intervals follows from Lemma~\ref{lem:stage} at stage
$\Gamma(\alpha,\beta)$.
\end{proof}

\section{Maximality and preservation}\label{sec:maximality}

\begin{proof}[Proof of Theorems~\ref{thm:main} and~\ref{thm:preservation}]
Work in $V$ and fix a dominating family $D_0\subseteq\om^\om$
of size $\omo$. Put
\[
 D_1=\{\max(f,\operatorname{id}):f\in D_0\},\qquad
 D=\{d\in\Delta:(\exists f\in D_1)\ d=^*f\}.
\]
Each function has only countably many finite modifications, so
$|D|\leq\omo$. Since $D$ is dominating, it is uncountable; hence
$|D|=\omo$. Enumerate $D$ as $\langle d_\alpha:\alpha<\omo\rangle$
and carry out the recursion of Section~\ref{sec:construction} in $V$.
Let $\A=\{a_\xi:\xi<\omo\}$ and the enumerations $\mathbf c^\alpha$
be those produced by this recursion, with the intervals and
integers chosen at each stage. We use the intervals $I^\alpha_{j_n}$
chosen in the recursion to prove maximality simultaneously in $V$
and in every outer model in which $D_0$ remains dominating.

Let $W$ be an outer model in which $D_0$ remains dominating;
we allow $W=V$. The family $D$ dominates pointwise in $W$.
Indeed, for $g\in\om^\om\cap W$, choose $f\in D_1$ with
$g\leas f$. The function $\max(f,g)$ is a finite modification
of $f$. Every such modification belongs to $V$, so
$\max(f,g)\in D$. Almost disjointness of $\A$ is also preserved,
since intersections of its members are unchanged finite sets.

Given $y\in[\om]^\om\cap W$, choose $\alpha<\omo^V$ with
$e_y\leq d_\alpha$. If $y$ has infinite intersection with a
member of $\A_\alpha$, there is nothing to prove. Otherwise its
trace $t_y$ with respect to $\mathbf c^\alpha$ belongs to
$\om^\om\cap W$. Choose $\beta<\omo^V$ with $t_y\leq d_\beta$.
Equation~\eqref{eq:traces} gives $y\subseteq G_{\alpha,\beta}$.
Here $G_{\alpha,\beta}$ is the same set in $V$ and in $W$, since
$G(\mathbf c,g)$ is defined arithmetically from $\mathbf c$ and $g$.

We check that the capture argument applies even when $y\notin V$.
Write $G=G_{\alpha,\beta}\in V$. Since $e_y\leq d_\alpha$,
Lemma~\ref{lem:intervals}, applied in $W$, shows that $y$ meets every
interval $I^\alpha_j$; since $y\subseteq G$, so does $G$. This
concerns fixed finite sets and is absolute between $V$ and $W$.
Thus, in $V$, stage $\xi=\Gamma(\alpha,\beta)$ used
$(\rho,H)=(\alpha,G)$, and
$a_\xi\supseteq G\cap I^\alpha_{j_n}$ for every $n$.
In $W$, the set $y$ meets each of these
ground-model finite sets. They are pairwise disjoint, so
$|y\cap a_\xi|=\om$.

This proves that $\A$ is MAD in $W$. Taking $W=V$ gives
$\aaa=\omo$, proving Theorem~\ref{thm:main}.
Finally, in any extension by an $\om^\om$-bounding forcing,
every function is eventually dominated by a ground-model
function, and hence by a member of $D_0$. Thus $D_0$ remains
dominating, and the same family $\A$ remains MAD, proving
Theorem~\ref{thm:preservation}.
\end{proof}

\section{An almost strongly separable family}\label{sec:ch}

Choosing the pointwise dominating sequence
$\langle d_\alpha:\alpha<\omo\rangle$ to be nondecreasing under $\leas$
makes the family produced by the same recursion almost strongly
separable. We first recall the relevant definitions.
For an almost disjoint family $\A$, let $\mathcal I(\A)$ be the
ideal generated by $\A$ and the finite sets. For an ideal
$\mathcal I$ on $\om$ containing the finite sets, write
$\mathcal I^+=\mathcal P(\om)\setminus\mathcal I$, and let
$(\mathcal I^{<\om})^+$ consist of the families
$X\subseteq[\om]^{<\om}\setminus\{\varnothing\}$ such that every
$B\in\mathcal I$ is disjoint from some $s\in X$.
The family $\A$ is \emph{almost strongly separable} if, for every
$X\in(\mathcal I(\A)^{<\om})^+$, some $a\in\A$ contains infinitely
many elements of $X$ \cite[Definition~5.2]{BGHR26}.
If one only requires some $B\in\mathcal I(\A)$, rather than some
$a\in\A$, to contain infinitely many elements of $X$, one obtains the
weaker \emph{Shelah--Stepr\=ans} property
\cite[Definitions~5.1 and~5.3]{BGHR26}.
A MAD family is \emph{tight} if, for every sequence
$\langle Y_n:n\in\om\rangle$ in $\mathcal I(\A)^+$, some
$B\in\mathcal I(\A)$ meets every $Y_n$ \cite[Definition~6.1(1)]{BGHR26}.
Every Shelah--Stepr\=ans MAD family is tight, and every tight MAD
family is Cohen-indestructible \cite[Section~6]{BGHR26}.

\begin{theorem}[$\dd=\omo$]\label{thm:cohen}
There is an almost strongly separable MAD family $\A$ of size $\omo$
that remains MAD in every outer model in which $\om^\om\cap V$
is dominating. In particular, $\A$ is Shelah--Stepr\=ans, tight
and Cohen-indestructible, and no $\om^\om$-bounding forcing
destroys it.
\end{theorem}

\begin{proof}
Work in $V$ and fix a dominating family
$\{g_\nu:\nu<\omo\}$. Since countably many functions are bounded,
we may recursively choose $f_\nu\in\om^\om$ with $g_\nu\leas f_\nu$ and
$f_\mu\leas f_\nu$ for $\mu<\nu$. Let $D_0=\{f_\nu:\nu<\omo\}$ and
define $D_1$ and $D$ as in Section~\ref{sec:maximality}.
Enumerate $D$, allowing repetitions, in consecutive blocks of order
type $\om$: the block $\langle d_{\om\cdot\nu+n}:n\in\om\rangle$
lists all finite modifications of $\max(f_\nu,\operatorname{id})$
that belong to $\Delta$. The resulting sequence dominates pointwise
and satisfies
\begin{equation}\label{eq:scale}
 d_\gamma\leas d_\delta\qquad(\gamma\leq\delta<\omo).
\end{equation}
Carry out the recursion of Section~\ref{sec:construction} with this
sequence. The proof of Section~\ref{sec:maximality} applies
unchanged, so $\A$ is MAD in every outer model in which $D_0$
remains dominating, that is, in which $\om^\om\cap V$ is dominating.

Fix $X\in(\mathcal I(\A)^{<\om})^+$. We shall find $\delta,\beta<\omo$
and sets $s_j\in X$ with $s_j\subseteq G_{\delta,\beta}\cap I^\delta_j$
for all sufficiently large $j$.
The interval capture property of Lemma~\ref{lem:construction} will
then place infinitely many of these sets in a single member of $\A$.
We first choose $\delta$ so that its intervals can accommodate
elements of $X$ avoiding progressively more members of $\A_\delta$;
we then use a trace bound to choose $\beta$.

For $\gamma<\omo$ and $k,m\in\om$, let $e_\gamma(k,m)$ be the least
value of $\max s$ among sets $s\in X$ disjoint from
$c^\gamma_0\cup\cdots\cup c^\gamma_k\cup[0,m)$.
Such sets exist by the positivity of $X$.
For each $\gamma$, choose $F(\gamma)<\omo$ such that
$e_\gamma(k,\cdot)\leas d_{F(\gamma)}$ for every $k$.
Such an index exists because the countably many functions
$e_\gamma(k,\cdot)$ have a common eventual bound and the sequence
$\langle d_\alpha:\alpha<\omo\rangle$ is dominating.
Choose a nonzero countable limit ordinal $\delta$ such that
$F(\gamma)<\delta$ whenever $\gamma<\delta$.

We claim that
\begin{equation}\label{eq:local-bound}
 e_\delta(k,\cdot)\leas d_\delta\qquad(k\in\om).
\end{equation}
Fix $k$.
The sets $c^\delta_0,\ldots,c^\delta_k$ are finitely many members
of $\A_\delta$, hence of $\A_\gamma$ for some $\gamma<\delta$, and
so they occur among $c^\gamma_0,\ldots,c^\gamma_{k'}$ for some $k'$.
Consequently, $e_\delta(k,\cdot)\leq e_\gamma(k',\cdot)
\leas d_{F(\gamma)}\leas d_\delta$ by $F(\gamma)<\delta$ and
\eqref{eq:scale}, proving \eqref{eq:local-bound}.

We now choose the sets $s_j$. Write $h=h_{d_\delta}$ and, by
\eqref{eq:local-bound}, choose integers $t(0)<t(1)<\cdots$ such that
$e_\delta(k,m)\leq d_\delta(m)$ whenever $m\geq t(k)$.
Choose $j_0$ with $h(j_0)\geq t(0)$. For $j\geq j_0$, put
$\kappa(j)=\max\{k:t(k)\leq h(j)\}$.
By the definition of $e_\delta$, there is $s_j\in X$ disjoint from
$c^\delta_0\cup\cdots\cup c^\delta_{\kappa(j)}\cup[0,h(j))$
with $\max s_j\leq d_\delta(h(j))$. Hence
$s_j\subseteq[h(j),d_\delta(h(j))]=I^\delta_j$.

Let $y=\bigcup_{j\geq j_0}s_j$. Since $\kappa(j)$ tends to infinity,
each $c^\delta_k$ meets only finitely many of the sets $s_j$.
Thus $y\cap c^\delta_k$ is finite for every $k$, and the trace
$t_y$ with respect to $\mathbf c^\delta$ is defined.
Choose $\beta$ with $t_y\leq d_\beta$. By \eqref{eq:traces},
$y\subseteq G_{\delta,\beta}$, and therefore
$s_j\subseteq G_{\delta,\beta}\cap I^\delta_j$ for every $j\geq j_0$,
as required.

In particular, $G_{\delta,\beta}$ meets all but finitely many
intervals $I^\delta_j$. By Lemma~\ref{lem:construction},
$a_{\Gamma(\delta,\beta)}$ contains $G_{\delta,\beta}\cap I^\delta_j$,
and hence $s_j$, for infinitely many $j\geq j_0$. These sets are
nonempty and lie in disjoint intervals, so they are distinct.
This proves almost strong
separability; the remaining assertions follow from the implications
above and the preservation property.
\end{proof}

\noindent\textbf{Use of AI.}
Over several weeks, the author worked on different parts of the argument
in parallel and submitted them for review to Anthropic's Claude models
(Fable 5, Opus 5, Fable 5.1, and Opus 5.5) and OpenAI's GPT models
(5.6 Sol and 6 Astra). The author proposed ideas, examples, and proof
strategies; the systems checked them, looked for counterexamples, and
suggested corrections, which the author adopted, modified, or discarded.
The two main combinatorial ideas---the simultaneous capture of candidates
through domination and intersection traces, and the system of coherent
interval reservations that preserves almost disjointness throughout the
recursion---are due to the author and were refined in this back-and-forth.
The argument as a whole, however, grew out of this process, and the human
and AI contributions to it cannot be precisely separated. This paper would
not exist without these AI systems. The author takes full responsibility
for all of its content, including the mathematical claims, proofs, and
references.

\end{document}